\documentclass[12pt]{amsart}

\usepackage{amssymb,amsfonts,amsmath}
\usepackage{enumerate}
\usepackage{mathrsfs}
\usepackage[pdftex]{graphicx}
\usepackage[section]{placeins}
\usepackage{subcaption}
\usepackage{hyperref}
\hypersetup{
    colorlinks=false, 
    linktoc=all,     
    linkcolor=blue,  
    linktocpage,
}
\usepackage{xcolor}
\usepackage{soul}
\usepackage{wrapfig}

\newtheorem{thm}{Theorem}[section]

\newtheorem{lem}[thm]{Lemma}

\theoremstyle{definition}
\newtheorem{defn}[thm]{Definition}

\theoremstyle{remark}

\makeatletter
\let\c@equation\c@thm
\makeatother
\numberwithin{equation}{section}

\newcommand{\R}{\mathbb{R}}

\newcommand{\C}{\mathbb{C}}

\title{Regular Polygon Surfaces}
\author{\href{http://www.dam.brown.edu/people/ialevy/}{Ian M. Alevy}}
\address[I. Alevy]{Division of Applied Mathematics, Brown University, Providence, RI 02912}
\email[I. Alevy]{ian\_alevy@brown.edu}

\date{\today}

\begin{document}

\begin{abstract}
A \emph{regular polygon surface} \(M\) is a surface graph \((\Sigma, \Gamma)\) together with a continuous map \(\psi\) from \(\Sigma\) into Euclidean 3-space which maps faces to regular Euclidean polygons. When \(\Sigma\) is homeomorphic to the sphere and the degree of every face of \(\Gamma\) is five, we prove that \(M\) can be realized as the boundary of a union of dodecahedra glued together along common facets. Under the same assumptions but when the faces of \(\Gamma\) have degree four or eight, we prove that \(M\) can be realized as the boundary of a union of cubes and octagonal prisms glued together along common facets. We exhibit counterexamples showing the failure of both theorems for higher genus surfaces.
\end{abstract}

\maketitle

\tableofcontents 

\section{Introduction}

 \begin{figure}
    \centering
    \begin{subfigure}[b]{0.4\textwidth}
        \includegraphics[width=\textwidth]{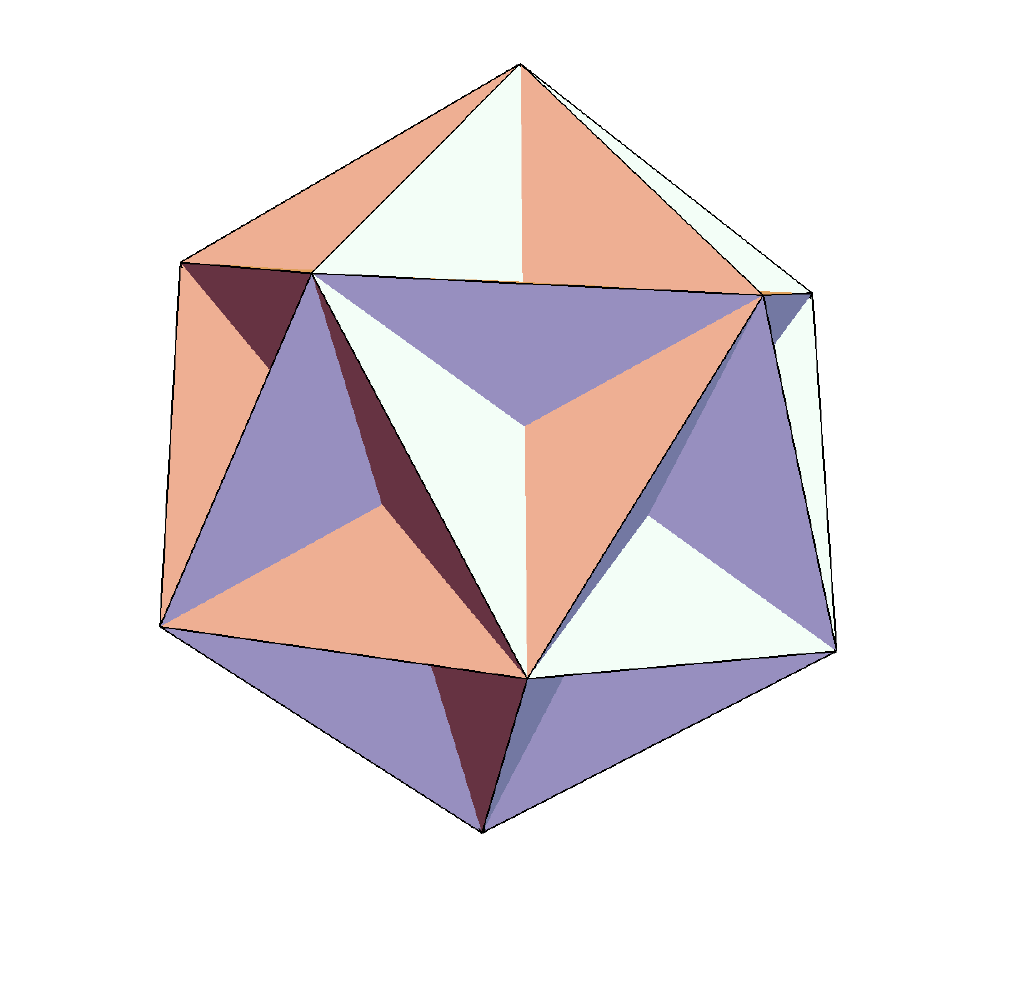}
        \caption{Great dodecahedron}
        \label{fig:great-dodec}
    \end{subfigure}
\qquad
    ~ 
    \begin{subfigure}[b]{0.4\textwidth}
        \includegraphics[width=\textwidth]{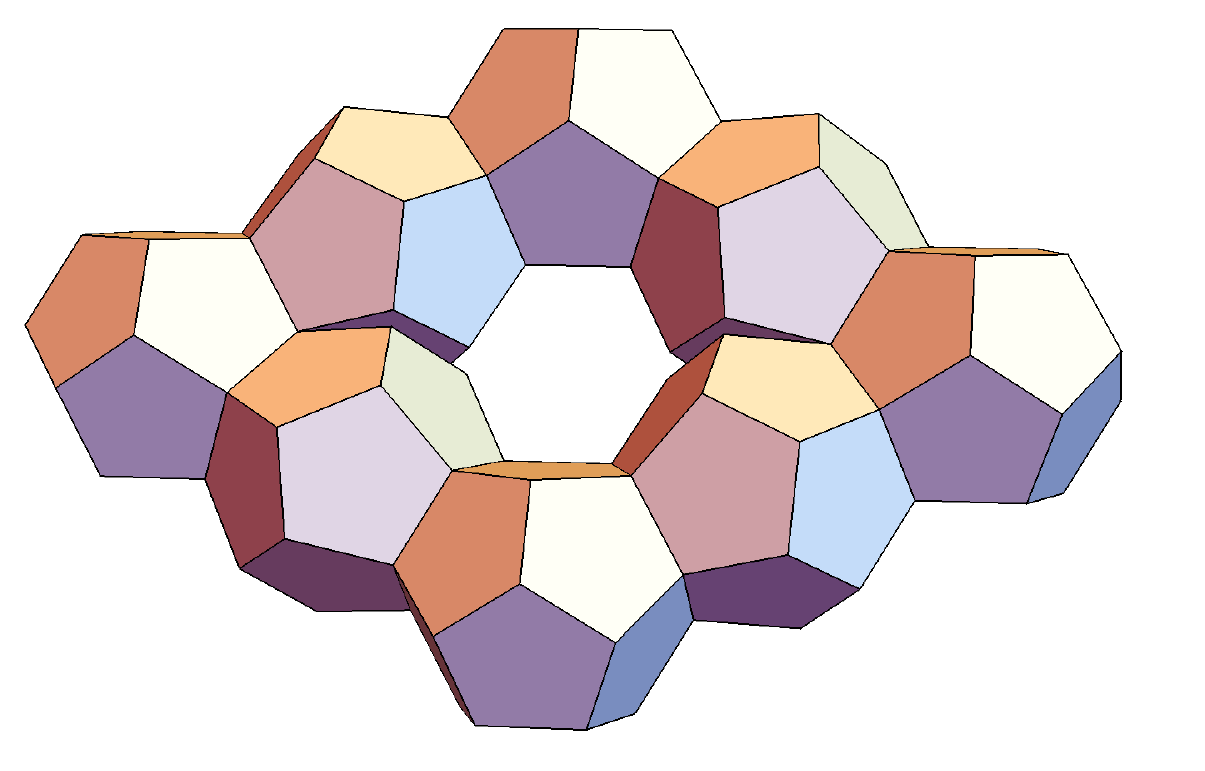}
        \caption{Dodecahedral torus}
        \label{fig:dodec-torus}
    \end{subfigure}
    ~ 

    \caption{Regular polygon surfaces with degree five faces}\label{fig:surfs}
\end{figure}

We study surfaces built by gluing regular and rigid Euclidean polygons together along their edges.
Recently surfaces built out of regular polygons with boundary have been used to build flexible metamaterials that can be deformed into various configurations \cite{o2017}. Very little is known about the space of shapes of these generalized polyhedra, called \emph{regular polygon surfaces} (RPSs) (see definition \ref{def:rps}),  which are neither convex nor symmetric. We prove that under certain assumptions on the genus and face degrees, RPSs can be realized as a union of Platonic solids glued along common facets. Before giving a rigorous definition of a RPS, we introduce some terminology from graph theory.

\begin{defn}\label{def:surface-graph}
	A \emph{surface graph} \((\Sigma, \Gamma)\) is a graph \(\Gamma\) embedded on a closed surface \(\Sigma\) in such a way that \(\Sigma\setminus \Gamma\) is a union of connected complementary components called \emph{faces} with each face homeomorphic to the \(2\)-cell. If in addition the closure of each face is homeomorphic to the closed \(2\)-cell, we call the surface graph a \emph{regular surface graph}. When the intersection of the closure of any two faces is either empty, a vertex in \(\Gamma\), or an edge in \(\Gamma\) we say that the surface graph is \emph{proper}. 
	\end{defn}	
The \emph{degree} of a face in a surface graph is the number of edges incident to that face.
	\begin{defn}\label{def:rps}
	Let \((\Sigma,\Gamma)\) be a finite, regular, and proper surface graph in which \(\Sigma\) is a genus \(g\) surface. A genus \(g\) \emph{regular polygon surface} (RPS) is a triple \((\Sigma, \Gamma, \psi)\) whose \emph{geometric realization} \(\psi: \Sigma \to \R^3\) is continuous and maps a face of degree \(k\) to a regular Euclidean \(k\)-gon with unit edge lengths. To rule out degenerate RPSs, we assume that the intersection of the image under \(\psi\) of adjacent faces in the graph is either one vertex or one edge and its two incident vertices. If all of the face degrees are contained in the set \(\{k_1, \ldots, k_n\}\), then we call the surface a \((k_1, \ldots , k_n)\)-RPS.
	\end{defn}
	We allow geometric realizations which are not \emph{embeddings}, i.e., \(\psi\) may not be injective and the geometric realization of the surface may have self-intersections. 
	
	In figure \ref{fig:surfs} we show two examples of RPSs. The familiar Platonic solids as well as the Kepler-Poinsot polyhedra \cite{c1963} are all examples of RPSs. One way to build more complicated RPSs is to \emph{glue} two RPSs together along a common facet when both surfaces have facets with the same number of incident edges. To be precise, suppose \(P\) and \(Q\) are RPSs which both have a face of degree \(n\) and let \(f_p\) and \(f_q\) denote the respective faces. After cutting out the interior of \(f_p\) from \(P\) and the interior of \(f_q\) from \(Q\) we can glue the two surface graphs together along their boundaries (in an orientation-reversing way). The new RPS inherits a geometric realization in the obvious way from the geometric realizations of the original two RPSs. An example of a RPS constructed in this fashion is the \emph{dodecahedral torus} shown in figure \ref{fig:dodec-torus}.

When the genus is low enough the space of RPSs is constrained and we are able to prove the following three theorems.

\begin{thm}\label{pent}
Every oriented genus \(0\) or \(1\), \((5)\)-RPS can be realized as the boundary of a union of dodecahedra glued together along common facets.
\end{thm}

\begin{thm}\label{pent-n}
The only possible oriented genus \(0\), \((5,7,8,9,10)\)-RPSs are those which can be realized as the boundary of a union of dodecahedra glued together along common facets.
\end{thm}

\begin{thm}\label{square-oct}
Every oriented genus \(0\), \((4,8)\)-RPS can be realized as the boundary of a union of cubes and octagonal prisms glued together along common facets.
\end{thm}

\begin{wrapfigure}{r}{0.5\textwidth}
  \begin{center}
    \includegraphics[width=0.48\textwidth]{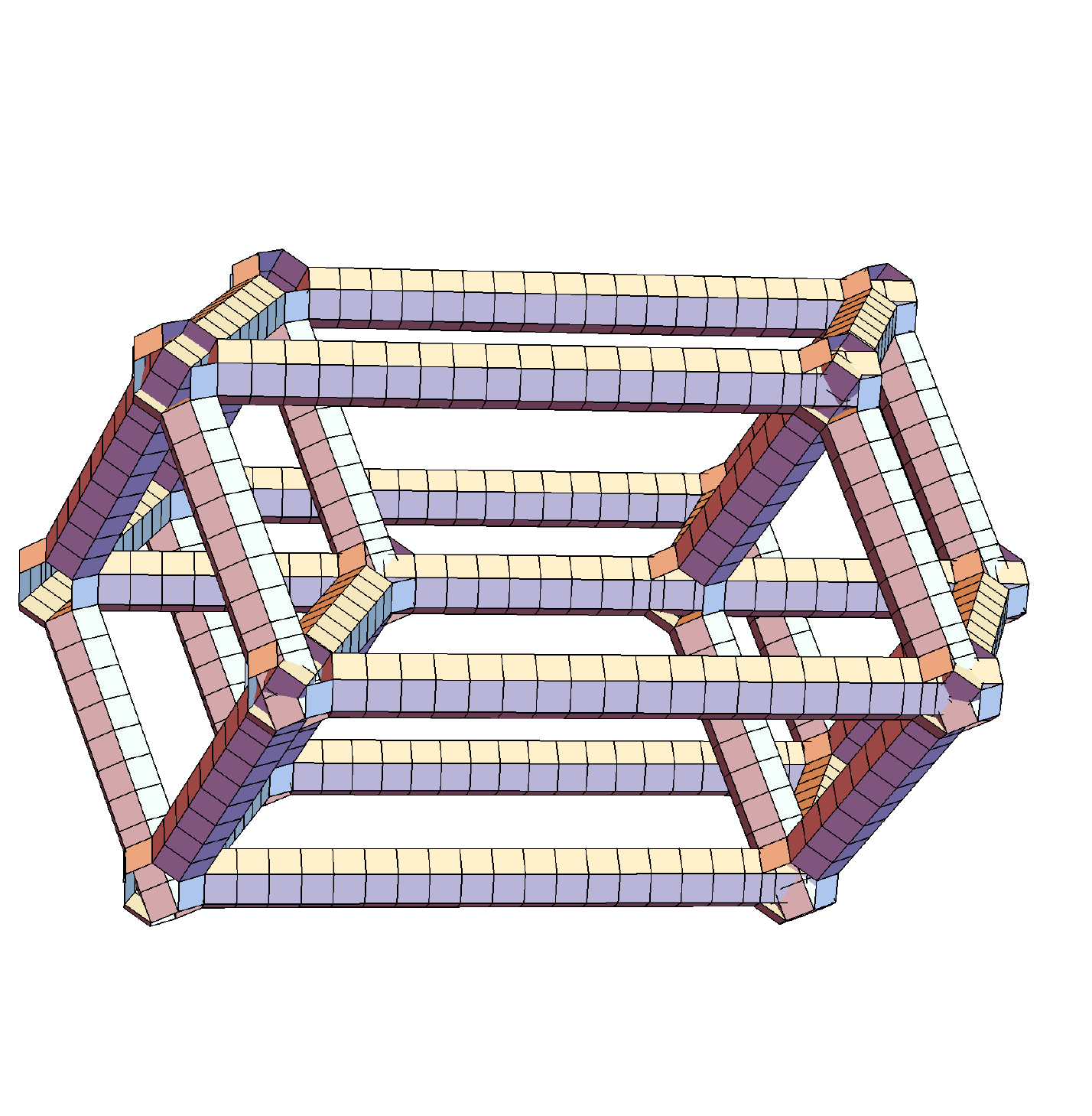}
\caption{A \((4)\)-RPS of genus 49}
\label{truncOctMol4ds}
  \end{center}
\end{wrapfigure}
Not all RPSs can be constructed by gluing together convex polyhedra. Figure \ref{truncOctMol4ds} shows a \((4)\)-RPS with genus 49 which cannot be realized as a union of cubes and prisms glued together. In section \ref{sec:counter} we explain how this surface is constructed and present two other examples of high genus RPSs which are not unions of cubes and prisms (figs. \ref{truncCubOctMol4d} and \ref{truncCubOctMolSqHex}). Note that the surfaces described in section \ref{sec:counter} cannot be embedded \(\R^3\).

The RPSs studied in this paper are examples of generalized polyhedra. While the five convex regular polyhedra known as \emph{Platonic solids} were described in Euclid's \emph{Elements}, there is no characterization of generalized polyhedra. New examples of polyhedra (in \(\R^3\)) are still being discovered (see \cite{gsw2014} for examples of regular polyhedra and \cite{gs2009} for examples of \emph{toroidal polyhedra}). Moreover, mathematicians have not reached a consensus on the definition of a generalized polyhedron. See \cite{g2003a} for a historical account of the study of polyhedra and a proposed definition of a generalized polyhedron.  Although it is generally known that there are 13 convex \emph{Archimedean polyhedra}, whether regularity is a ``local'' or ``global'' condition has resulted in a mathematical error in many enumerations of these objects \cite{g2009}. 

A few examples of generalized polyhedra were known classically. The Kepler-Poinsot great dodecahedron (figure \ref{fig:great-dodec}) is a genus four RPS with faces of degree five which is not a union of convex polyhedra since its vertex figures are the nonconvex star polygons known as pentagrams. It can be constructed in two steps by first \emph{stellating} (extending the faces symmetrically to form a new polyhedron) the dodecahedron to obtain the small stellated dodecahedron, then dualizing the polyhedron \cite{c1963}. The great dodecahedron was first depicted in a 1568 etching by Amman (see figure \ref{fig:jamnitzer}) of an engraving made by Jamnitzer \cite{j1568}.

	\begin{wrapfigure}{l}{0.4\textwidth}
  \begin{center}
    \includegraphics[width=0.38\textwidth]{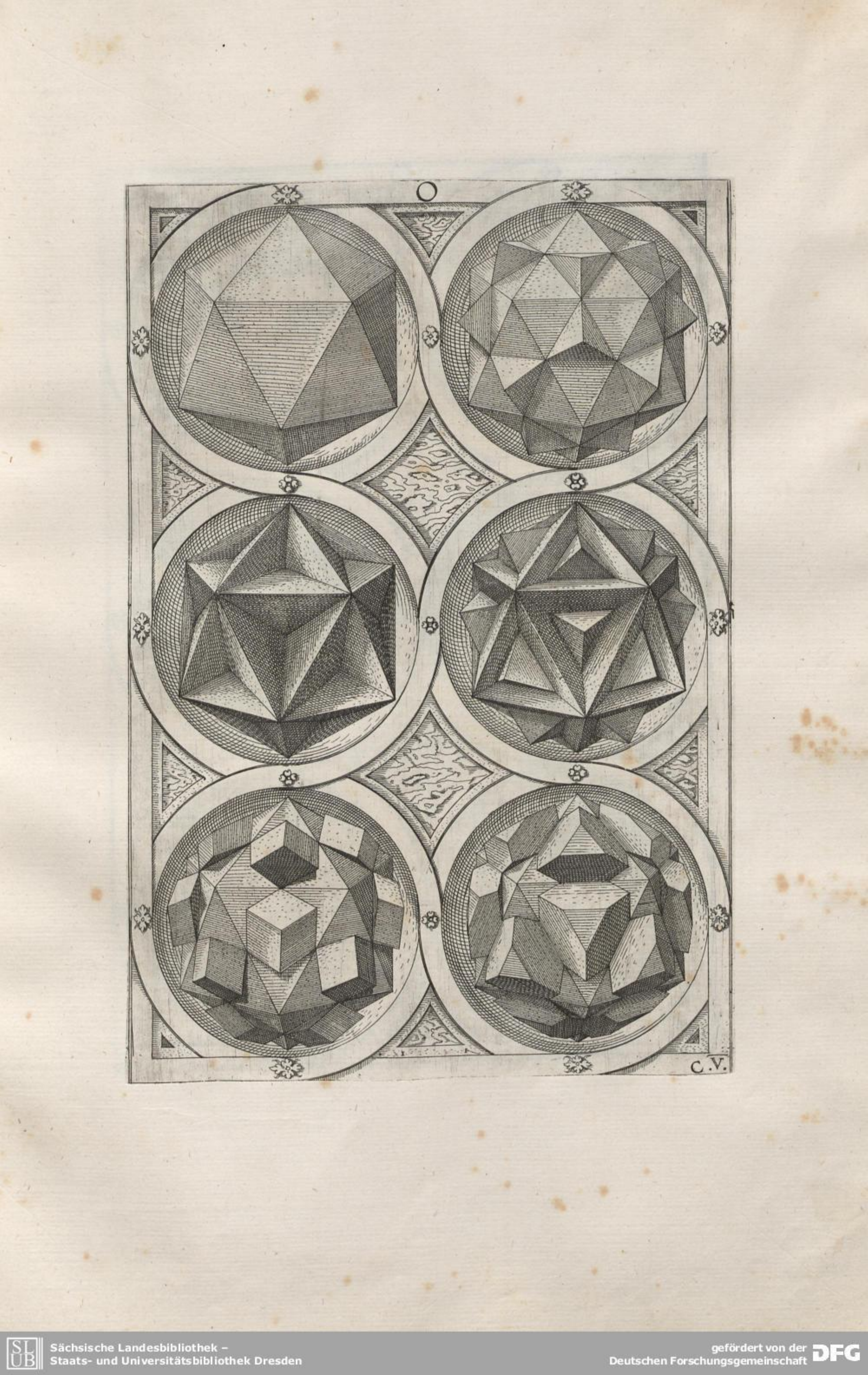}
  \end{center}
  \caption{The Renaissance etching showing the earliest known depiction of the great dodecahedron \cite{j1568}}
  \label{fig:jamnitzer}
\end{wrapfigure} 

There has been some recent work on low genus polyhedra with rectangular faces. Donoso and O'Rourke \cite{do2001} proved that a polyhedron, of genus at most one with rectangular faces, has dihedral angles which are all integer multiples of \(\pi/2\). In addition they constructed a genus seven polyhedron with rectangular faces whose dihedral angles are not integer multiples of \(\pi/2\). Their result was later extended to genus two polyhedra with rectangular faces by Biedl et al. \cite{bcd2002}. Thurston developed a global theory that describes triangulations of the sphere with at most \(6\) triangles around a vertex \cite{t1998}. He found a natural bijection between these triangulations and a quotient space of a discrete lattice in \(\C^{1,9}\). See \cite{s2015} for a readable introduction to Thurston's paper which provides alternative proofs of the main theorems in \cite{t1998}.

RPSs also have applications to statistical mechanics. In the lattice formulation of quantum gravity, physicists are naturally led to an infinite dimensional integral over the space of Riemannian metrics. By approximating a manifold by piecewise linear manifolds, such as RPSs, with fixed edge lengths, one can replace the integral by a discrete sum, vastly simplifying the problem \cite{d1992}. Sampling large random piecewise linear manifolds is an important aspect of this theory \cite{ab2014}. It is conjectured that the associated metric space converges to the Brownian map. See \cite{lm2012} for a definition of the Brownian map and a survey of results about large random planar maps. Surface graphs which support a family of different geometric realizations can be used as a model for random surfaces. Our results imply that certain surface graphs do not have a non-trivial space of geometric realizations.

Although much is known about convex polyhedra in \(\R^n\) (see \cite{a2005} or \cite{g2003c}) and convex ideal polyhedra in \(\mathbb H^3\) (see \cite{r1996}), their techniques do not apply to the inherently nonconvex surfaces we study in this paper.

\nocite{gm1963,g1937,k1996n}

In order to prove the three main theorems we use inductive arguments that rely on a procedure for simplifying RPSs by removing certain subgraphs. Just as we can build more complicated RPSs by gluing two RPSs together, there is an inverse process where we can simplify a RPS by removing certain subgraphs and replacing them by others. We call the process \emph{polyhedral surgery} and it is defined as follows. Let \(P\) be a RPS with data \((\Sigma_1, \mathcal G_1, \psi_1)\) and \(Q\) a RPS with data \((\Sigma_2, \mathcal G_2, \psi_2)\). 
Suppose both RPSs contain cycles of length \(n\), call them \(C_1 \) and \(C_2\) respectively. Label the vertices in \(P\) along \(C_1\) by \(v_1,\ldots,v_n\) and the vertices in \(C_2\) by \(w_1,\ldots ,w_n\). In addition, suppose there exists an isometry \(G\) of \(\R^3\) such that \(G\circ \psi_1(v_i)= \psi_2(w_i)\). Cutting \(\Sigma_1\) along the cycle disconnects it into two hemispheres, \(H_P^1\) and \(H_P^2\) and likewise for \(\Sigma_2\) with hemispheres \(H_Q^1\) and \(H_Q^2\). Now we can glue \(H_P^1\) to \(H_Q^1\) along their respective boundaries to form a new RPS with geometric realization defined by 
\[ \psi(x) = \begin{cases} G \circ \psi_1(x)\qquad & \text{ if } x \in H_P^1 \\ \psi_2(x) \qquad & \text{ if } x \in H_Q^1 \end{cases}.\]
It is possible that there is a face \(f\) in \(H_P^1\) and a face \(g\) in \(H_Q^1\) that are adjacent in the surface graph of the new RPS and \(\psi(f)\) and \(\psi(g)\) intersect in more than just one edge. When two faces intersect in this manner we call them \emph{dangling faces}. However, a slight modification of \(C_1\) to include \(f\) prevents this from occurring.

Polyhedral surgery can always be performed when the geometric realization of the RPS formed by gluing the hemispheres \(H_P^2\) and \(H_Q^1\) forms a convex polyhedron. As an example we consider a case when the two hemispheres form a cube. Let \(P\) be a RPS containing the subgraph as shown on the left-hand side of figure \ref{sqsurg} and \(Q\) a RPS containing the subgraph shown on the right-hand side of the same figure. We can cut \(P\) along the cycle in green, remove the hemisphere containing \(v_1\) and glue in the hemisphere of \(Q\) which contains the subgraph shown in the right-hand side of figure \ref{sqsurg}. The resulting RPS has the same genus as \(P\) but has two fewer faces. 

\begin{figure}
  \begin{center}
    \includegraphics[width=\textwidth]{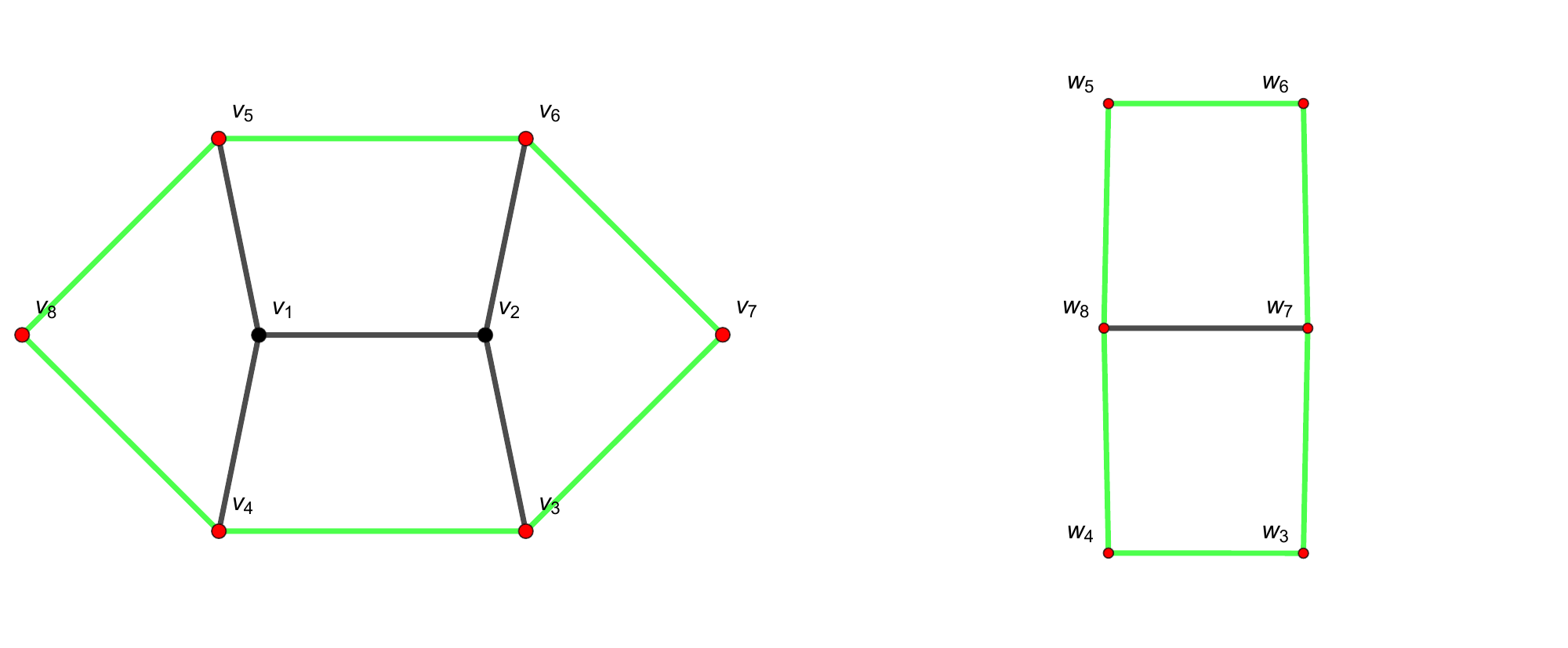}
\caption{Surgery on a cube}
\label{sqsurg}
  \end{center}
\end{figure}


\section{\((5)\)-RPSs}\label{sec:rps5}

When the degree of each face in a RPS is large and the genus is low, geometric constraints impose some rigidity on the structure of the geometric realization of the surface. In this section we use a discrete form of the Gauss-Bonnet theorem to prove theorem \ref{pent}. 
\begin{thm}[Discrete Gauss-Bonnet Theorem]
For a RPS \(P\) with Euler characteristic \(\chi\) and curvature \(k_v\) at each vertex \(v\) of \(P\) we have
\begin{align*} \sum_{v\in P}^n k_v  = 2 \pi \chi \end{align*}
where the \emph{vertex curvature} \(k_v\) is \(2\pi\) minus the sum over all faces containing \(v\) of the interior angle at \(v\) in the geometric realization of the face.
\end{thm}
A proof of this theorem can be found in many places including \cite{s2011}. While the Gauss-Bonnet theorem is not normally proved for the RPSs we study in this paper, the extension to our case is straightforward. One way to prove the theorem is by triangulating the RPS, counting the contribution to the curvature from each face \(\pi\), and applying Euler's formula for a graph on a surface with Euler characteristic \(\chi\).

We will find it more useful to assign curvature to faces instead of vertices. The \emph{facial curvature} \(k_f\) associated to face \(f\) is given by
\[ k_f=\sum_{v\in f}^n \frac{  k_{v}}{d_{v}} \]
where the sum is over all vertices incident to \(f\) and \(d_v\) is the degree of vertex \(v\).

For the remainder of this section we restrict our attention to pentagonal RPSs, those with all faces of degree five. Before proving theorem \ref{pent} we first prove a useful lemma. 
\begin{lem}\label{pos-curv face}
Let \(f\) be a face of a RPS with all faces of degree five. If \(f\) has positive facial curvature then each vertex is of degree three. Moreover, if \(f\) has one negative curvature vertex then \(f\) has non-positive facial curvature. 
\end{lem}

\begin{proof}
Since \(f\) has positive facial curvature there must be a vertex incident to \(f\) with positive vertex curvature. The curvature at a vertex of degree \(d\) is \(\pi/5 (10-3d)\) which is only positive if the vertex has degree three. This vertex contributes \(\pi/5 (10/d -3)\) to the facial curvature of each face containing it. If vertices with degrees \(d_1, \ldots d_5\) are incident to a face then its facial curvature is 
\[\frac{\pi}{5} \sum_{i=1}^5 \left( \frac{10}{d_i}- 3 \right)=-3 \pi + \pi \sum_{i=1}^5 \frac{2}{d_i}.\]
This function is monotonically decreasing as a function of the degrees. If one vertex has degree \(6\) then it is non-positive. Checking the finitely many cases in which each vertex has degree less than or equal to \(6\) we find that the facial curvature is only positive when the face has at least \(4\) degree three vertices and a fifth vertex with degree at most \(5\).

Next assume that the face \(f\) has four vertices of degree three and a fifth vertex \(v\) with degree either four or five. Let \(g\) and \(h\) be the two faces incident to \(v\) which share an edge with \(f\). Since \(g\) shares an edge with \(f\) it must also share two vertices with \(f\) and since \(f\) only has one vertex with degree more than three, \(g\) and \(f\) must also share a degree three vertex. Thus the dihedral angle between \(g\) and \(f\) in the geometric realization is fixed to be that of the dodecahedron, and likewise for the dihedral angle between \(f\) and \(h\) in the geometric realization. This implies that the geometric realizations of \(g\) and \(h\) intersect along an edge. Therefore three is the maximum degree of \(v\). From this proof we find that when \(f\) has one negative curvature vertex it must have a second and any face with at least two negative curvature vertices has facial curvature of at most zero. 
\end{proof}


\begin{proof}[Proof of theorem \ref{pent}]
First we prove theorem \ref{pent} for genus zero RPSs. Suppose that the set of counterexamples to the theorem is nonempty. Our RPSs are assumed to be finite thus there exists a lower bound on the number of faces in a surface which is an element of the set of counterexamples. Let \(n\) be this lower bound and \(P\) a member of the set of counterexamples with \(n\) faces. Let \((\Sigma, \Gamma, \psi)\) denote the data of \(P\). We use polyhedral removal surgery to construct a RPS with fewer than \(n\) faces whose realization is not a union of dodecahedra, therefore contradicting the assumption of minimality.

By assumption, \(P\) has genus zero and total curvature \(4\pi\) thus contains a face with positive facial curvature. Lemma \ref{pos-curv face} states that the degree of every vertex incident to a face with positive facial curvature is three. Let \(f\) be a face with positive facial curvature. To simplify notation, we label faces that share an edge with \(f\) as \emph{first-generation} faces and faces that share an edge with first-generation faces as \emph{second-generation} faces. If a first-generation face had positive curvature, then it would have to share a degree three vertex with some second-generation face. Let \(C_1\) be the cycle in \(\Gamma\) bounding the seven faces consisting of \(f\), the first-generation faces and the face in the second-generation which shares a degree three vertex with a first generation face. Cut the surface along \(C\) into two hemispheres \(H_P^1\) and \(H_P^2\), where \(H_P^1\) is the hemisphere with seven faces. Since all faces in \(H_P^1\) are connected by the same type of degree three vertices it can be realized as a hemisphere of a dodecahedron. 

Let \(Q\) be a genus zero RPS with twelve faces. As can be seen from counting the curvature at every vertex, its geometric realization is a dodecahedron. Cut it along a curve \(C_2\) into two hemispheres such that one hemisphere has five faces and the other has seven. Label the hemisphere with five faces \(H_Q^1\) and the other \(H_Q^2\). Using polyhedral removal surgery we can glue \(H_Q^1\) and \(H_P^2\) along their boundaries to form a genus zero RPS \(P'\) with \(n-2\) faces. Since \(P\) cannot be realized as a union of dodecahedra and the faces we removed can be realized as a part of a dodecahedra, the new surface \(P'\) cannot be realized as a union of dodecahedra. However, \(P'\) has \(n-2\) faces which contradicts the assumption that \(n\) was the lower bound on the number of faces in a RPS which cannot be realized as a union of dodecahedra. It is possible that after gluing the hemispheres together, two adjacent faces have the same geometric realization. However, we can resolve this issue by changing \(C_1\) so that one of these faces is in \(H_P^1\). Likewise we modify \(C_2\) so that the boundaries of \(H_Q^1\) and \(H_P^2\) agree. 

Now we argue by contradiction to establish that the surface must contain a face with positive facial curvature with an adjacent face that also has positive facial curvature. Assume that no face in the first-generation has positive facial curvature. No face in the second-generation can have positive facial curvature either. Otherwise it would have all degree three vertices thus share a degree three vertex with a first-generation face and we could apply the same argument as in the preceding paragraph to construct a counterexample to the theorem with \(n-2\) faces. Since each first generation face has two degree three vertices and the facial curvature is a monotonic function of the degrees, the facial curvature can be maximized by maximizing the number of degree three vertices. Each face has negative curvature so there are at most three degree three vertices. The only geometrically realizable configuration with three degree three vertices, subject to the constraint that the face has negative facial curvature, is the configuration with three vertices of degree three that are in different orientations. However in this case the remaining two vertices must have degree at least five giving this configuration curvature \(-\pi/5\), otherwise the surface would not have a valid geometric realization. If the face has two degree three vertices then its facial curvature is maximized with three degree four vertices and this configuration has facial curvature \(-\pi/6\). Since facial curvature is a monotonic function, any configuration with less than two degree three vertices will have less curvature than the configuration with two degree three vertices and three degree four vertices. 

The central face \(f\) has facial curvature \(\pi/3\) which gives the region including \(f\) and the first generation faces, total facial curvature \(-2\pi/3\). Every positive curvature face must be contained in a region with total facial curvature at most \(-2\pi/3\), and these regions must be disjoint because second generation faces also have negative facial curvature. Thus, \(-2\pi s/3\) is an upper bound on the total curvature of the surface where \(s\) is the number of positive curvature faces. This contradicts our assumption that the surface has genus zero and positive total curvature. Therefore the RPS must have at least one face with positive facial curvature which is adjacent to a face with positive facial curvature. Using polyhedral removal surgery we can always build a counterexample to the theorem with fewer than \(n\) faces.

Finally, we extend the result to genus one RPSs. Arguing in the same manner as the genus zero case, suppose the set of genus one counterexamples to theorem \ref{pent} is non-empty and let \(n\) be a lower bound on the number of faces of a RPS in this set. Let \(P\) be an element of the set of counterexamples with \(n\) faces. Since the total curvature of \(P\) is zero, we divide the proof into two cases. In the first case, \(P\) has at least one face with positive facial curvature. The same argument in the preceding paragraph shows that \(P\) must have a region of seven contiguous faces on which we can use polyhedral surgery to build a counterexample with \(n-2\) faces, thus contradicting the assumption that \(n\) is a lower bound on the number of faces in a counterexample. In the second case, every face of \(P\) has zero facial curvature. Recall that the curvature of a face with vertices of degrees \(d_1,\ldots, d_5\) is 
\[-3 \pi + \pi \sum_{i=1}^5 \frac{2}{d_i}.\]
This sum is negative when at least one of the vertices has degree greater than six. Checking the (finitely many) cases with each \(d_i\leq 6\) we find that the sum is zero either when the face has four degree three vertices and one degree six vertex or when the face has three degree three vertices and two degree four vertices.  A RPS cannot have a face with four vertices of degree three and one vertex of degree six. Two adjacent faces incident to such a vertex would have 2-dimensional intersection in the surface's realization thus violating one of the conditions in the definition of a RPS. For the remainder of the proof we assume that every face has three degree three vertices and two degree four vertices.

In each face the two degree four vertices must be adjacent in order for the surface to have a geometric realization. This condition severely restricts the combinatorics of the underlying surface graph. In figure \ref{gzcase} we show a subgraph of a surface graph for which the three faces \(f_1, f_2\) and \(f_3\) satisfy this requirement on the degrees. Notice that there are three degree three vertices in the interior of the cycle shown in green, thus we can apply polyhedral removal surgery on the green cycle to reduce the number of faces in the surface. The hemisphere we glue in has five faces and comes from a hemisphere of a dodecahedron.

\begin{figure}
  \begin{center}
    \includegraphics[scale=0.4]{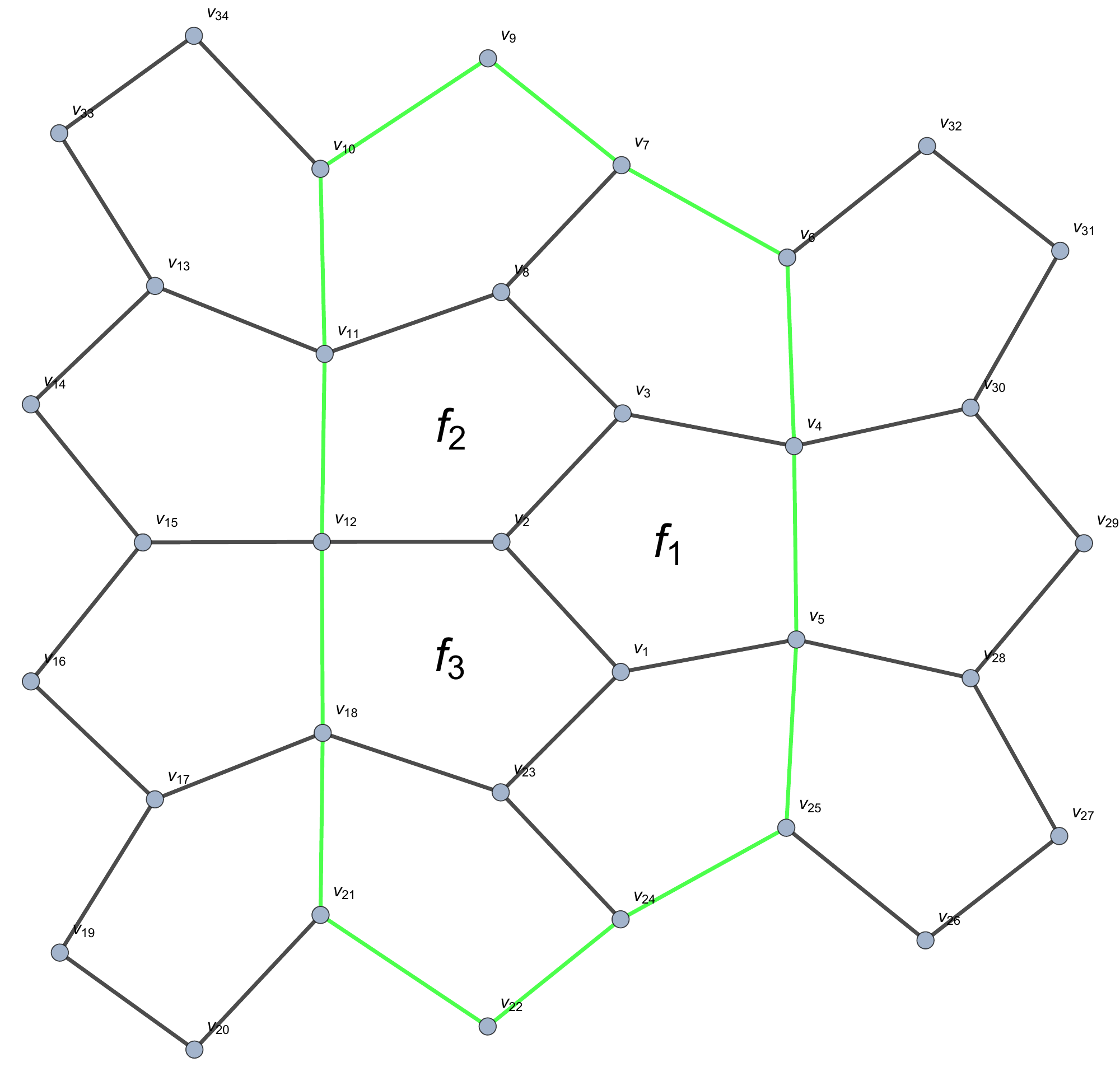}
\caption{Subgraph of the surface graph of a RPS in which every face has zero facial curvature}
\label{gzcase}
  \end{center}
\end{figure}

\end{proof}

\section{\((5,7,8,9,10)\)-RPSs}

In this section we restrict our attention to \((5,7,8,9,10)\)-RPSs and prove theorem \ref{pent-n}. Our assumption that a RPS has a geometric realization in \(\R^3\) places a restriction on the types of degree three vertices that may be present. A degree three vertex is only geometrically realizable when it has non-negative curvature. The curvature of a vertex at which two degree five faces and one degree \(n\) face meet is \((n-10)\pi/5n\) so the configuration is only realizable if \(n \leq 10\), justifying our restriction on the maximum degree of a face. Moreover, if two degree seven (or higher) faces meet at a vertex then the curvature is negative. Thus every degree three vertex is formed by the intersection of at least two degree five faces and one other face which may degree larger than five.

We exclude degree six faces because the existence of large combinatorial spheres with regular pentagonal and hexagonal faces, such as the truncated icosahedron, present an obstacle to our methods. Our methods are local arguments and these large combinatorial spheres imply that we must examine neighborhoods with many faces. Furthermore, vertices at which three degree six faces meet have zero vertex curvature which implies that the surface may have large regions of faces with zero facial curvature in between positive curvature faces. Nevertheless we conjecture that any genus zero RPS with faces of degree five or higher can be realized as a union of dodecahedra and truncated icosahedra glued together along common facets.

Before proving theorem \ref{pent-n} we introduce notation for the different types of vertices a face may have. A vertex at which \(k\) faces of degree \(m\) and \(l\) faces of degree \(n\) meet is denoted by \(m^k,n^l\). For a face with vertices \(v_1, \ldots v_n\) we use the product notation \(({m_1}^{k_1},{n_1}^{l_1}) \cdots ({m_t}^{k_t},{n_t}^{l_t}) \) to indicate that vertex \(v_i\) is of the form \(({m_i}^{k_i},{n_i}^{l_i})\). We always assume the vertices are ordered cyclically. 

In order for a RPS to have a geometric realization only certain vertex combinations on a face are allowed. For example, a vertex of the form \((5^2,7)\) cannot be adjacent to a vertex of the form \((5^3)\). At a degree three vertex the degree of the faces incident to the vertex determine the dihedral angles between the images of the faces under the geometric realization. In table \ref{table:angles} we record the dihedral angles between faces incident to vertices with non-negative vertex curvature. Since the dihedral angles are different for different type vertices, we find that two degree three vertices can only be adjacent if both vertices have the same type. 

\begin{center}
	\begin{table}
    \begin{tabular}{| l | l | l | }
    \hline
    Vertex type & \(5-5\) Dihedral angle &  \(5-n\) Dihedral angle.  \\ \hline
    \((5^3)\) & \(116.57^{\circ}\)  & \(116.57^{\circ}\) \\
    \((5^2,7)\) & 142.65 & 132.43   \\
    \((5^2,8)\) &  152.54 & 141.67   \\
     \((5^2,9)\) & 162.27 & 153.22   \\
     \((5^2,10)\) & 180 & 180   \\
    \hline
    \end{tabular}
    \caption{Dihedral angles between faces incident to a \((5^2,n)\) vertex} \label{table:angles}
    \end{table}
\end{center}

Suppose a face \(f\) is incident to two degree three vertices \(v\) and \(w\), both of which are adjacent to a third vertex \(u\). The dihedral angles between any two of the three faces incident to a degree three vertex are determined by the vertex type (see table \ref{table:angles}). Let \(g\) be the face incident to both \(u\) and \(v\) (which isn't \(f\)) and \(h\) the face incident to both \(u\) and \(w\) (which isn't \(f\)). If \(v\) and \(w\) do not have the same vertex type then the dihedral angle between \(g\) and \(h\) will not be one of those listed in table \ref{table:angles}, implying that the degree of \(u\) is at least \(4\). However, the vertex cannot have degree \(4\) because a regular polygon cannot be adjacent to both \(g\) and \(h\). Therefore the vertex must have degree at least \(5\).

Now we prove two lemmas which we use to prove theorem \ref{pent-n}.

\begin{lem}\label{verts n face}
Let \(f\) be a face of a RPS of degree \(n\) with \(n\geq 7\). If \(f\) has positive facial curvature then every vertex incident to \(f\) has the form \((5^2,n)\). Moreover, if \(f\) has one vertex with negative vertex curvature then \(f\) has negative facial curvature.
\end{lem}

\begin{proof}
The only positive curvature vertices are those of the form \((5^2,n)\) with vertex curvature \(\pi \left(2/n  - 1/5 \right)\). The curvature of a degree four vertex incident to \(f\) is at most \(\pi \left(2/n  - 1/5 \right)\) where the vertex has configuration \((5^3,n)\). If a face has \(k\) vertices of degree greater than \(3\) and \(n-k\) vertices of degree \(3\) then its facial curvature is \( -\pi (2n^2+4kn -20n + 5k)/30n \) which is negative when \(k >1 \) for \(n\geq 7 \). 

Assume that \(f\) has \(n-1\) vertices of type \((5^2,n)\) and one vertex \(v\) which may have degree greater than 3. Let \(g\) and \(h\) be the two faces adjacent to \(f\) which are also incident to \(v\). Under the geometric realization of the surface, \(g\) and \(h\) intersect along an edge because the degree \(3\) vertices determine the dihedral angle between these two faces. If \(v\) had degree \(4\) then a face would be adjacent to both \(g\) and \(h\) but then the realization of this face could not be a regular polygon because two of its edges have the same geometric realization. This implies that \(v\) has degree at least five. However, if \(v\) had degree \(5\) then the geometric realization of two of the faces incident to \(v\) would overlap which is a contradiction. We find that \(v\) has degree at least \(6\). 

A vertex of degree \(6\) has the largest vertex curvature in configuration \((5^5-n)\) with vertex curvature \(2\pi (-1+1/n)\). If a face has \(n-1\) vertices of type \((5^2-n)\) and one vertex of degree \(6\) then its facial curvature is at most \(-(n-1)(n-5)\pi/15 n\) which is negative for \( n>5\). Therefore when \(f\) has positive facial curvature every vertex must be of type \((5^2,n)\). Furthermore, this argument implies that a face with one negative curvature vertex must have a second vertex with negative vertex curvature and thus have negative facial curvature.
\end{proof}

\begin{lem}\label{pent-face-pos}
In a \((5,7,8,9,10)\)-RPS of genus zero, there exists a face of degree five with positive facial curvature.
\end{lem}

\begin{proof}
Suppose every face of degree five has non-positive curvature and let \(f\) be a degree \(n\) face with positive facial curvature. By lemma \ref{verts n face}, each vertex incident to \(f\) has configuration \((5^2,n)\) and so \(f\) is adjacent to \(n\) degree five faces, all of which have non-positive facial curvature. Each face in the first generation has at least two adjacent vertices of type \((5^2,n)\). As previously discussed, degree \(3\) vertices with different types cannot be adjacent and are separated by a vertex of degree at least \(5\). A configuration with two different types of degree three vertices has facial curvature at most \(\pi (160-47n)/75n\) when it is of type \((5^2,n)^2(5^4,n)(5^3)(5^4,n)\). Each first generation has at most \(4\) degree three vertices since the remaining vertex is incident to two degree \(n\) faces. The configuration with only one type of degree three vertex with the largest facial curvature is \((5^2,n)^4(5^2,n^2)\) with facial curvature \(\pi(110/n-17)/30 \). Any other facial configuration has less facial curvature and we find that the most curvature a face in the first generation can have is in the configuration \((5^2,n)^4(5^2,n^2)\).

The sum of the facial curvature from \(f\) and its \(n\) first generation faces is \(13 \pi/3- 19 \pi n /30\) which is negative for \(n \geq 7\). None of the faces in \(f\)'s second generation can have positive curvature because each either has degree five or is a degree \(n\) face incident to a vertex of degree \(4\) of the form \((5^2,n)^4(5^2,n^2)\) and as a result has negative facial curvature.

Thus, by summing the curvature over all faces with positive facial curvature and their first generation faces we find that the surface has negative total curvature. This contradicts the assumption that the surface has genus zero and we conclude that there exists a degree five face with positive facial curvature.
\end{proof}

\begin{proof}[Proof of theorem \ref{pent-n}]
The proof is similar to the proof of theorem \ref{pent}. Suppose that the set of counterexamples to the theorem is non-empty and let \(n\) be the lower bound on the number of faces of an element in the set. Let \(P\) be a counterexample with \(n\) faces. By lemma \ref{pent-face-pos} there exists a degree five face \(f\) with positive facial curvature. Since \(f\) has positive curvature, every vertex incident to \(f\) has positive vertex curvature. Moreover, \(f\) is incident to an odd number of vertices, implying that all vertices incident to \(f\) have the form \(5^3\). Thus all faces in the \(f\)'s first generation have degree five. If the second-generation faces were all degree five, then we could use the argument from the proof of theorem \ref{pent} to construct a counterexample with \(n-2\) faces. Thus, there exists a face with positive curvature such that a face in its second generation has degree larger than \(5\). Without loss of generality, assume that \(f\) is this face. If a face in \(f\)'s first generation had positive curvature then we could use polyhedral surgery, as in the proof of theorem \ref{pent}, to construct a counterexample to theorem \ref{pent-n} with \(n-2\) faces. The same argument implies that no degree five face in the second generation can have positive curvature. 

Let \(g\) be a face with degree greater than five in the second-generation. This face shares two vertices with a face \(h\) in the first generation. Since \(h\) has degree five and two vertices of the form \(5^3\), the two vertices it shares with \(g\) must have degree at least four. Since \(g\) has two degree four vertices its facial curvature is \(-\frac{\pi  (n-5) (n-1)}{15 n}\) which is negative for \(n \geq 6\). The vertex configuration of a first-generation face which maximizes the facial curvature of the face is the one with the fewest degree \(n\) vertices. The first generation faces contribute the most curvature with three of type \((5^3)^2(5^4)^3\), one of type \((5^3)^2(5^3,n)(5^2,n)(5^4)\), and one of type \((5^3)^2(5^4)(5^3,n)(5^4)\). The sum of the facial curvature of \(f\) and its first generation faces is \((5-2n )\pi/(3n)\) which is negative for \(n \geq 4 \). Since no second-generation face has positive curvature, the sum over all positive curvature faces, of the facial curvature of a face and its first-generation neighbors, gives an upper bound on the total curvature of the surface. However, this upper bound is negative which contradicts our assumption that the surface has genus zero.
\end{proof}

\section{\((4,8)\)-RPSs}

RPSs with faces of degree four or eight have both vertices with zero curvature and faces with zero facial curvature. Since positive curvature and negative curvature faces can be separated by large regions of zero curvature faces, we cannot use the curvature of local regions to rule out certain configurations as in the previous sections. However, these surfaces have additional structure which is not present in RPSs with faces of degree five. Let \(P\) be a RPS with data \((\Gamma, \Sigma, \psi)\). The geometric realization of each face in the graph is composed of pairs of parallel edges. 
\begin{defn}
A \emph{band} \(B_{e,f}\) is a simple cycle in the dual graph \(\bar{ \mathcal G}\) of \(\mathcal G\) starting at edge \(e\) and face \(f\) of \(\mathcal G\) with the property that the geometric realizations of the primal edges associated to consecutive edges in the cycle are parallel translates (in \(R^3\)) of each other. 
\end{defn}
The \emph{dual graph} \(\mathcal{ \hat G} \) of \(\mathcal G\) is the graph whose vertices correspond to faces of \(\mathcal G\) and where two vertices in \(\mathcal{\hat G}\) are adjacent exactly when the corresponding faces in the primal graph \(\mathcal G\) share an edge. 

We can cut \(\Sigma\) along the edges in \(\mathcal G\) bounding a band \(B_{e,f}\). When the RPS has genus zero the cut disconnects \(\Sigma\) into two hemispheres \(H_1\) and \(H_2\), and an annulus. If every face in the band has degree four, then we can glue the two hemispheres together by identifying pairs of boundary edges, and their incident vertices, which were incident to the same face in the band \(B_{e,f}\) to form a new surface \(P'\). The geometric realization \(\psi'\) of \(P'\) is
\[\psi'(x)=\begin{cases} \psi(x)-e \qquad & \text{if } x\in H_1, \\ \psi(x) \qquad& \text{if } x \in H_2\end{cases}\]
where \(\psi\) is the realization of the original surface. We call the process of removing a band and gluing two hemispheres together \emph{band surgery}. The new surface \(P'\) satisfies all conditions of being a RPS except for one: two adjacent faces may have the same image under \(\psi\). Removing every such pair of dangling faces forms an actual RPS with the same genus as \(P\). If every face in a genus zero RPS has degree four then we can use band surgery to prove the following theorem. 

\begin{thm}\label{thm: square}
Every genus zero RPS with faces of degree four can be realized as a union of cubes glued together along common facets. 
\end{thm}

\begin{proof}
We use complete induction on the number of faces in the surface to prove the theorem. Since the total curvature is \(4\pi\) and the most curvature a vertex can have is \(\pi/2\), there are at least eight vertices in the surface. Likewise, the facial curvature can be as large as \(2\pi/3\) when all vertices are degree three. Thus there are at least six faces in the surface. The only possible geometric realization of a RPS with six faces is that of a cube. This proves the base case of the theorem. For our inductive hypothesis we assume the theorem is true for all surfaces with fewer than \(n\) faces with \(n \geq 6 \). Let \(P\) be a RPS with \(n\) faces. Let \(B_{e,f}\) be a band in the surface through face \(f\) with edges that are parallel transports of \(e\) (in their geometric realizations). Since all faces in the surface have degree four we can use band surgery to remove \(B_{e,f}\) and form a new surface with fewer than \(n\) faces. Remove all pairs of dangling faces until what remains is a RPS which we call \(P'\). By induction it is a union of cubes. Let \(\gamma_1\) and \(\gamma_2\) denote the boundary edges of \(B_{e,f}\) in \(P\) which are identified to form \(P'\). Since \(P'\) is a union of cubes, the arc \(\gamma_1\) can be realized by a \emph{cycle} in a surface built by gluing cubes together along common facets. A \emph{cycle} is a sequence of alternating edges and vertices starting and ending at the same vertex such that each edge is incident to the two vertices preceding and succeeding it in the sequence and without repeated edges. Splitting \(P'\) along \(\gamma_1\) and inserting a new layer of cubes forms a surface \(\tilde P\). For each pair of dangling faces that was removed we glue in a cube, possibly identifying faces of neighboring cubes. The resulting RPS \(Q\) can be realized as a union of cubes. Since \(Q\) has the same genus as \(P\) and the surface graphs of the two surfaces are isomorphic, we conclude that \(P\) can be realized as a union of cubes.
\end{proof}

When faces of degree eight are also present in the surface, the surface may not have a band in which every face has degree four. Band surgery doesn't work on bands with degree eight faces because after cutting the band out and gluing the two hemispheres together, the resulting surface may not have a valid geometric realization. However, we can still use bands to help us characterize the structure of these surfaces. A path in the dual graph has a \emph{turning point} at a primal face \(f\) if the two faces adjacent to \(f\) in the path intersect \(f\) along edges which are not parallel translates of each other in the geometric realization of the surface. A \emph{band bigon} is a simple cycle in the dual graph with exactly two turning points. Every band bigon is formed by two intersecting bands which intersect at the primal faces corresponding to the two turning points. The faces of the bigon are the primal faces corresponding to the dual vertices of the bigon. The geometric realization of the two turning points are faces in the RPS which lie in parallel planes in \(\R^3\) since both faces contain parallel transports of the edges determining the bands. At a turning point the dot product between the two unit vectors determining the bigon is either \(0, 1/\sqrt{2}\) or \(-1/\sqrt{2}\). Since the geometric realization of the RPS sends every edge of the graph to a unit vector in \(\R^3\), we will often abuse notation and identify an edge with its corresponding unit vector in \(\R^3\).

Cutting \(\Sigma\) along the boundary of a bigon disconnects the surface into two disks and an annulus. If one of the disks does not contain any bigons then the bigon is said to be \emph{minimal}. We call the subset of \(\Sigma\) corresponding to the union of the annulus and the disk which doesn't contain any bigons, the \emph{interior} of the bigon. The \emph{strict interior} of the bigon is just the disk which contains no bigons. It is an easy consequence of the Jordan curve theorem that on a genus zero RPS, the bands forming a minimal bigon pass through adjacent edges of the faces corresponding to the turning points of the bigon (see lemma \ref{lem: no sq oct}).

After a careful analysis of all minimal bigons, we show that the there are only a few possible geometric realizations of the interior of a minimal bigon. For any minimal bigon, the arc \(\gamma\) bounding the interior can be realized as a cycle on a surface built out of a union of cubes and octagonal prisms glued together along common facets.
This result is established in a sequence of lemmas and is crucial in the proof of theorem \ref{square-oct}. The first lemma in the sequence states that the turning points of a minimal bigon are either both squares or both octagons. To prove this we use the elementary fact that any band which passes through the interior of a minimal bigon must cross both bands forming the bigon. This is an easy corollary of the Jordan curve theorem combined with the assumption of minimality. A genus zero RPS cannot have a simple cycle in the dual graph with exactly one turning point (a \emph{monogon}) because no band can pass through two edges of a face which are not parallel to each other.

\begin{lem}\label{lem: no sq oct}
On a genus zero RPS, no minimal bigon can have a square at one turning point and an octagon at the other.
\end{lem}

\begin{proof}
Suppose two distinct bands \(B_v\) and \(B_h\) form a minimal bigon with one turning point a degree four face and the other a degree eight face. Label the faces \(S\) and \(O\) respectively. Since the two bands cross in a degree four face the dot product of their corresponding unit vectors must satisfy \(|v \cdot h|=1\). Thus the geometric realizations of \(S\) and \(O\) are parallel to the plane spanned by \(v\) and \(h\). The edges on \(O\) which are parallel transports of \(v\) and \(h\) cannot be adjacent which implies that there must be an edge between them and a band through this edge which passes through the interior of the minimal bigon. However, the RPS is topologically a sphere so by the Jordan curve theorem this band must exit the bigon and in the process cross either \(B_v\) or \(B_h\). In either case this contradicts the assumption of minimality.
\end{proof}

We are now able to classify bigons based on the types of their turning points. A minimal square bigon has squares at each of its two turning points and similarly a minimal octagon bigon has octagons at each of its two turning points.

\begin{lem}\label{lem: int oct}
Suppose \(B_v\) and \(B_h\) are two bands forming a minimal octagon bigon on a genus zero RPS. Two bands, neither of which are part of the bands forming the bigon, cannot cross in the interior of this minimal bigon.
\end{lem}

\begin{proof}
Let \(B_h\) and \(B_v\) be the two bands forming the minimal octagon bigon and let \(O_1\) and \(O_2\) be the two octagons at the turning points of the bigon. As previously noted, the bands of the bigon must pass through adjacent edges of the octagons at the turning points. Since the geometric realizations of \(O_1\) and \(O_2\) are regular octagons with unit edge lengths, the angle between \(v\) and \(h\) is fixed so that \(|h \cdot v | = 1/\sqrt{2}\). By changing coordinates we may assume \(v=(1,0,0)\) and \(h=1/\sqrt{2} (1,1,0)\). Since any band through the bigon must exit, the dot product of the unit vector associated to any band through the bigon with \(v\) or \(h\) is either \(0, 1/\sqrt{2}\) or \(-1/\sqrt{2}\).

First we show that the faces along \(B_h\) and \(B_v\) in the minimal bigon in between \(O_1\) and \(O_2\) all have degree four. Suppose there were an octagon on the boundary of the minimal bigon and without loss of generality that it's on \(B_v\). Let \(a,b\) and \(c\) denote the directions of the geometric realizations of the consecutive edges on this octagon with bands through them that enter the bigon. Since the realization of this face is a regular octagon we must have the following relations 
\[ | a \cdot v|=|c \cdot v| = |a \cdot b| = | b \cdot c |= 1/\sqrt{2} \text{ and } |b \cdot v| =0.\]
However, one of the edges of the octagon is a parallel transport of \(v\) and since the realization of the octagon is a plane these equations cannot all be satisfied. For \(|b \cdot v|=0\) implies \(b=(0,b_2, \pm \sqrt{1-b_2^2})\) but \(B_b\) must cross \(B_h\) which implies \(b_1=0 \text{ or } \pm 1\). Then there are six possibilities for \(a\), 
\[ a= 1/\sqrt{2}(1, \pm 1,0), 1/\sqrt{2}(1,0, \pm 1) \text{ or } 1/\sqrt{2}(0,1 \pm 1)\]
but none of these directions are allowed because \(B_a\) must cross \(B_h\). 

Now suppose two bands cross in a face in the interior of the bigon. Since this face cannot lie on \(B_h\) or \(B_v\), there are two edges \(e \) and \(g\) on the face with \(e \cdot g=0\). By the assumption of minimality \(B_e\) crosses both \(B_h\) and \(B_v\). Likewise \(B_g\) crosses both \(B_h\) and \(B_v\) as well. Since there are no octagons along \(B_h\) and \(B_v\), these crossings must all occur at degree four faces. Thus,
\[ | e\cdot h| = |g \cdot h|= |e \cdot v| = |  g\cdot v|=0\]
but there do not exist vectors in \(\R^3\) which satisfy both the above equations and \(e\cdot g =0\).

\end{proof}
The preceding lemma implies that the geometric realization of a minimal octagon bigon is in fact part of an octagonal prism. In other words, every face in the interior of the minimal bigon has degree four except for the two degree eight turning points. The dihedral angles between degree four faces are either \(\pi\) or \(3\pi/4\) and between degree eight and degree four faces the angles are \(\pi/2\). Moreover, every face in the interior of the bigon lies on one of the bounding bands and all bands that cross through the bigon are parallel transports of the same unit vector.

A minimal bigon is bounded by two cycles one of which is not adjacent to a face in the interior of the bigon. Let \(C\) be this cycle. Cutting \(\Sigma\) along \(C\) disconnects the surface into two hemispheres \(H^1\) and \(H^2\). Assume that \(H^1\) is the hemisphere containing the interior of the minimal bigon. As shown in the preceding paragraph, we can glue a hemisphere from a RPS which can be realized as an octagonal prism to \(H^1\) resulting in a RPS whose realization is an octagonal prism. Gluing the other hemisphere to \(H^2\) forms a new RPS with data \((\Sigma' , \mathcal G', \psi'\) in which \(\mathcal G'\) has two fewer degree eight faces than the original graph \(\mathcal G\). By construction, \(\Sigma'\) is homeomorphic to \(\Sigma\) so this operation does not change the genus of the surface. The geometric realization is defined from the original geometric realization \(\psi\) as \(\psi'(x)=\psi(x)\) for \(x\in H^2\) and extended by linearity to the complement \(\Sigma' \setminus H^1\). This operation is a special case of polyhedral surgery and we refer to it as \emph{octagon removal surgery}. It is possible that octagon removal surgery creates dangling faces, but after modifying \(C\) to include one of these faces there will be no dangling faces in the new surface. For later reference we record the content of this paragraph as the following lemma. 

\begin{lem}\label{lem: oct red}
Octagon removal surgery removes two degree eight faces from any genus zero RPS with a minimal octagon bigon.
\end{lem}

Square bigons are potentially more complicated, however in some cases they turn out to be very simple to analyze. 

\begin{lem}\label{lem: sq bigon}
If a genus zero RPS has a minimal square bigon without any degree eight faces on the boundary of the bigon, then the dihedral angle between any two faces in the interior of the bigon, in the realization of the surface, is \(\pi\). In particular, the realization of a minimal square bigon that does not contain an octagon consists of four faces of a rectangular prism. 
\end{lem}

\begin{proof}
Let \(B_a\) and \(B_b\) be the two bands forming the minimal square bigon. Suppose that two adjacent faces have a dihedral angle which is not \(\pi\) and let \(B_c\) and \(B_d\) be the two bands that pass through these two faces. Since both \(c\) and \(d\) are orthogonal to \(a\) and \(b\), \(c\) must be parallel to \(d\) which implies that the dihedral angle between the faces is \(\pi\). 

Now, suppose there is a face strictly in the interior of the bigon. Let \(B_c\) and \(B_d\) be the bands that cross at this face. Since both bands must pass through both sides of the bigon this creates four mutually perpendicular vectors in \(\R^3\) which is a contradiction. Thus the bigon consists of two bands that cross at two squares and every face in the interior of the bigon lies on one of the two belts, \(B_a\) or \(B_b\).
\end{proof}

When a minimal square bigon satisfies the conditions of \ref{lem: sq bigon}, its realization consists of four facets of a prism and we can use polyhedral removal surgery to remove these four facets and replace them by the other two facets of the prism in a manner analogous to octagon removal surgery. Polyhedral removal surgery applied to a prism reduces the number of faces in the surface's surface graph by \(2\). We call this special case of polyhedral surgery, \emph{prism removal surgery}. 

\begin{lem}\label{lem: edge dirs}
Let \(f\) be a face in the interior of a minimal square bigon and let \(e\) be a unit vector determined by the geometric realization of an edge incident to \(f\). Choose coordinates so that the bigon is formed by two bands, \(B_v\) and \(B_h\), with \(h=(1,0,0)\) and \(v=(0,1,0)\). Then the vector \(e\) is a parallel translate of one of eight possible unit vectors:
\[\frac{1}{\sqrt{2}}\left (1,\pm 1,0 \right), \frac{1}{\sqrt{2}} \left(1,0, \pm 1\right)), \frac{1}{\sqrt{2}} \left(0,1,\pm 1\right)) \text{ or }\left (0,0,\pm 1\right)).\]
\end{lem}

\begin{proof}
Let \(B_e\) be the band through the face \(f\) starting at edge \(e\). Since the bigon is minimal, \(B_e\) must pass through both \(B_v\) and \(B_h\) and it must cross at either a square or an octagon. Thus the angle between \(e\) and \(v\) is either \(\pi/2\), \(3\pi/4\), \(5\pi/4\) or \(3\pi/2\). Likewise for the angle between \(e\) and \(h\). Therefore \(e\) is a parallel translate of the eight directions listed in the statement of the theorem.
\end{proof}

\begin{lem}\label{lem: adj bands}
Suppose there is a degree eight face on the boundary of a minimal square bigon on a genus zero RPS and let \(B_a,B_b\) and \(B_c\) denote the other three bands through this face in cyclic order. Each face on \(B_a\) in the interior of the bigon is adjacent to a face on \(B_b\). Similarly, each face on \(B_c\) in the interior of the bigon is adjacent to a face on \(B_b\).
\end{lem}

\begin{proof}
Choose coordinates so that the square bigon is formed by two bands, \(B_v\) and \(B_h\), with \(h=(1,0,0)\) and \(v=(0,1,0)\). Let \(O\) be the degree eight face in the statement of the theorem and without loss of generality assume that it lies on \(B_v\). Since the surface is topologically a sphere, by the Jordan curve theorem \(B_a, B_b\) and \(B_c\) must cross \(B_h\). Since \(a,b,c\) and \(v\) are the directions of four consecutive edges of a face whose realization is a regular octagon, we have 
\[|a\cdot b| = | b \cdot c|= |c \cdot v| =|a \cdot v| = 1/\sqrt{2}\]
and
\[|a\cdot c| = | b \cdot v|= 0.\]
Moreover \(B_a,B_b\) and \(B_c\) all pass through \(B_h\), so the dot product of each with \(h\) is either \(0\), \(1/\sqrt{2}\) or \(-1/\sqrt{2}\). From lemma \ref{lem: edge dirs} we know that there are eight possible choices for the directions of the edges. The only possible directions for \(a,b,\) and \(c\) that satisfy all of these conditions are
\[
a = 1/\sqrt{2}(0,1,-1), \
b=(0,0,1), \text{ and }
c= 1/\sqrt{2} (0,1,1).\
\]

Now suppose that there is a face in the interior of the minimal square bigon between \(B_b\) and \(B_c\). This face could have degree four or eight, but in both cases there are two bands through this face determined by unit vectors which are orthogonal to each other. Let \(B_x\) and \(B_y\) denote these two bands. Since the bigon is minimal, \(B_x\) must cross \(B_c\) or \(B_a\), and likewise for \(B_y\). Of the eight possible directions for \(x\) and \(y\), the only directions consistent with the crossing condition is \(x=1/\sqrt{2}(0,1,-1)\) and \(y=1/\sqrt{2}(0,1,1)\). Computing \(x \cdot h\) we find that \(B_x\) crosses \(B_h\) at a degree eight face. Let \(O_1\) be this face and let the directions of the edges of the realization be \(a',b',\) and \(c'\). The fourth direction is a parallel translate of \(B_h\) because \(O_1\) is on \(B_h\). A similar calculation to the one determining the directions \(a,b\), and \(c\) shows that the unit vectors determining these directions are
\[
a' = 1/\sqrt{2}(1,0,-1), \
b'=(0,0,1), \text{ and } \
c'= 1/\sqrt{2} (1,0,1).
\]
Likewise \(B_y\) also crosses \(B_h\) at a degree eight face, which we label \(O_2\). An identical argument shows that the same unit vectors as above determine the directions of the realizations of the edges of \(O_2\). However, computing the dot products we find \(B_x\) crosses \(O_1\) along the edge in the direction of \(b'\) and \(B_y\) crosses \(O_2\) along the edge in the direction of \(b'\). This is a contradiction because \(x\) and \(y\) are orthogonal directions.
\end{proof}

There are two remaining operations to define before we can prove theorem \ref{square-oct}. In the proof we analyze all minimal bigons and use polyhedral surgery to decrease either the number of faces in the surface or the number of degree eight faces in the surface. When a minimal square bigon has degree eight faces along its boundary, the bigon can be very complex to analyze. The two operations we introduce are the \emph{cube flip} and the \emph{prism flip}. Both allow us to decrease the number of faces in the bigon, thereby reducing the complexity of the bigon.

First, we define the cube flip. Let \(h\) be a turning point of a minimal square bigon and let \(f\) and \(g\) be the two faces in the interior of the bigon which are adjacent to \(h\). If all three faces have degree four, then in their realization they form three faces of a cube. In an operation we call a \emph{cube flip}, we replace these three faces in the RPS by the three faces that form the other half of the cube. Figure \ref{cflip} shows a subgraph of a surface graph and the effect of a cube flip on this subgraph. The geometric realization of the new RPS is defined by extending the geometric realization of the original surface to the three faces \(f',g'\) and \(h'\) linearly so that the realization of each face is a Euclidean square with unit edge lengths. 

\begin{figure}
  \begin{center}
    \includegraphics[scale=0.6]{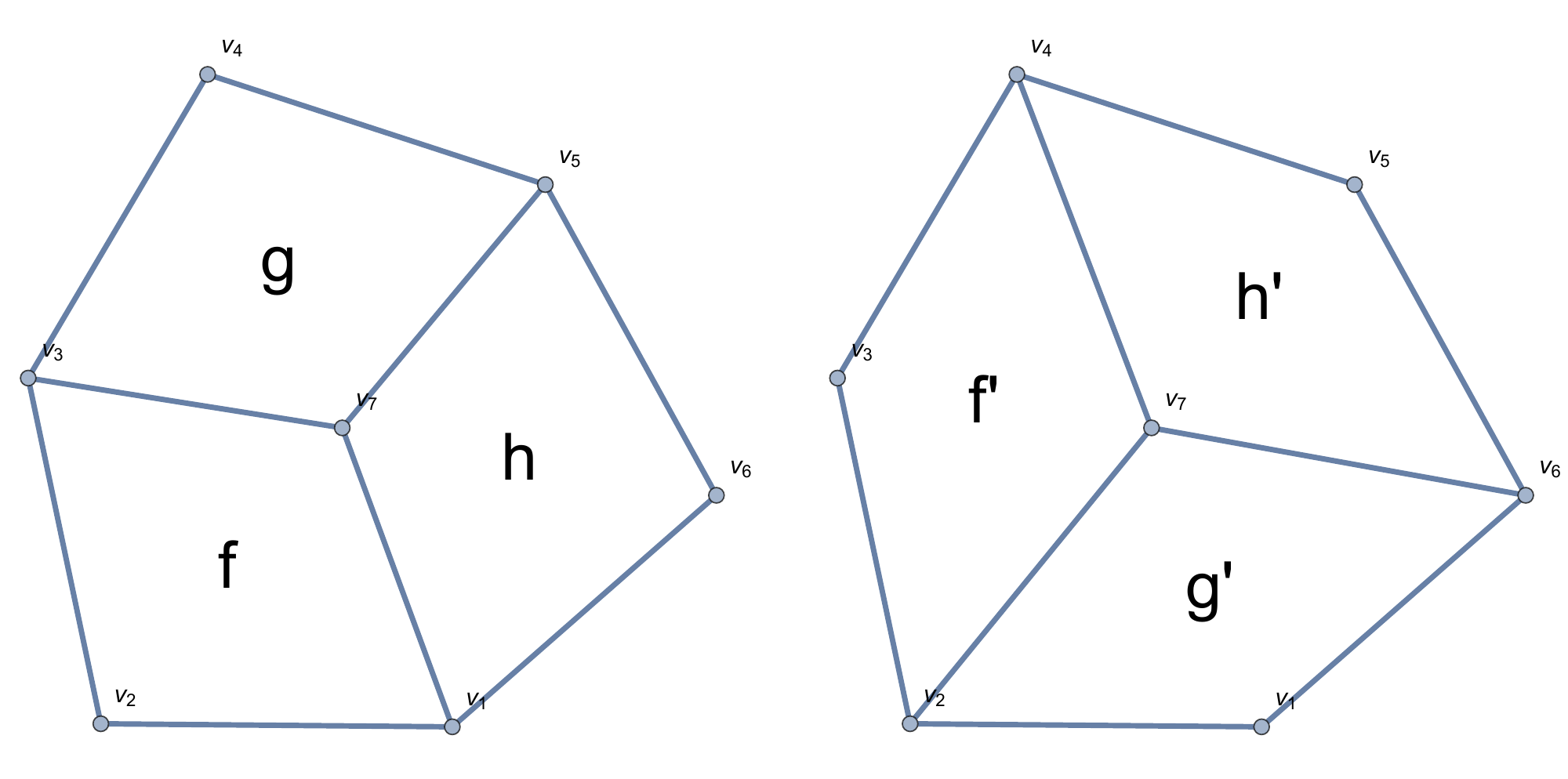}
\caption{Cube flip}
\label{cflip}
  \end{center}
\end{figure}

Suppose that there exists a minimal bigon and that a cube flip can be performing on one of the turning points. Let \(B_h\) and \(B_v\) denote the bands determining the bigon and \(f_1\) and \(f_2\) the two faces corresponding to the bigon's turning points. After a cube flip at \(f_1\), the bigon formed by \(B_h\) and \(B_v\), which passes through \(f_2\), also passes through one of the new faces formed by the cube flip. The new bigon is still a minimal bigon but it contains fewer bands through its interior than the original minimal bigon. 

The prism flip is defined similarly to the cube flip. In the prism flip we replace five faces in the RPS, whose realization forms part of an octagonal prism, with the five faces that form the other half of the prism. We skip a formal definition of the prism flip because it is so similar to the cube flip. A prism flip that involves the face at the turning point of the bigon reduces the number of bands that cross through the interior of the bigon. Finally, all of the tools are in place to prove theorem \ref{square-oct}.

\begin{proof}[Proof of theorem \ref{square-oct}.]
Let \(P\) be a RPS with data \((\Sigma, \Gamma, \psi)\). We use induction on the number of degree eight faces in the RPS. Let \(n\) be the number of degree eight faces in the surface. The base case, \(n=0\), is theorem \ref{thm: square}. Assume that the theorem is true for any RPS with fewer than \(n\) faces. Since the surface has a finite number of faces, there is always at least one minimal bigon. There are two cases depending on the type of the bigon. First, assume that the minimal bigon is an octagon bigon. Lemma \ref{lem: oct red} explains how octagonal removal surgery applied to this bigon produce a new RPS \(P'\) with two fewer degree eight faces than \(P\). By the induction hypothesis this surface can be realized as a union of cubes and prisms. Since the bigon that was removed can be realized as part of an octagonal prism, we can glue an octagonal prism to \(P'\) to form a surface which has the same surface graph and genus as \(P\). Thus \(P\) can be realized by a union of cubes and prisms. 

In the second case, the minimal bigon is a square bigon. Suppose that this bigon is determined by two bands, \(B_h\) and \(B_v\), and has a turning point at a face \(f\). Let \(g\) and \(h\) be the two faces in the interior of the minimal bigon which are adjacent to \(f\). If both \(g\) and \(h\) have degree four then we can use a cube flip to decrease the number of bands through this bigon. If one of \(g\) and \(h\) has degree four and the other degree eight, then by lemma \ref{lem: adj bands} there are five faces, including \(f\), whose realization is a part of an octagonal prism. Thus we can always use a cube flip or a prism flip to decrease the number of bands through the bigon. 

Since the bigon has finitely many faces, after finitely many flips the bigon will have one face which is adjacent to both turning points. If this face has degree four then we can remove a cube with polyhedral removal surgery and decrease the number of degree four faces in the surface by \(2\). If the face has degree eight then we could remove an octagonal prism with polyhedral removal surgery and decrease the number of degree eight faces in the surface by \(2\). This procedure reduces the number of degree four faces in the surface monotonically. However, we cannot remove all of the degree four faces from the surface. The surface has genus zero and thus has positive curvature vertices. At least two squares meet at every positive curvature vertex. Thus after removing finitely many faces from the surface, we must reach a surface \(P'\) in which the only minimal bigons are octagon bigons. Removing this octagon bigon and applying the induction hypothesis we find that the surface can be realized as a union of cubes and prisms. Since at every step we have removed either a part of a cube or a prism, we can glue these cubes and prisms back to \(P'\) to form a surface with the same surface graph and genus as \(P\). Therefore \(P\) can be realized as a union of cubes and prisms.

\end{proof}

\section{Examples of higher genus RPSs}\label{sec:counter}

In this section we construct three examples of high genus RPSs which are not unions of convex polyhedra. This can be seen by the absence of certain faces in the surfaces. All of the examples in this section can be constructed in two steps. First, place convex polyhedra at the vertices of a 3-cube or a 4-cube. Second, remove certain faces from each polyhedron and connect the boundary components using prisms with matching boundary components. Whether a 3-cube or a 4-cube is used depends on the structure of the convex polyhedra.

Figure \ref{truncOctMol4d} shows a genus 49 surface whose faces have degree four. It is constructed out of truncated octahedra placed at the vertices of a 4-cube and connected by hexagonal prisms. All hexagonal faces have been removed from the constituent truncated octahedra and hexagonal prisms. The Euler characteristic of the surface is
\begin{align*}
\chi & = 16\cdot 2-16\cdot 8-64\cdot 6+64\cdot 6 = -96.
\end{align*} 

\begin{figure}[h]
  \begin{center}
    \includegraphics[width=\textwidth]{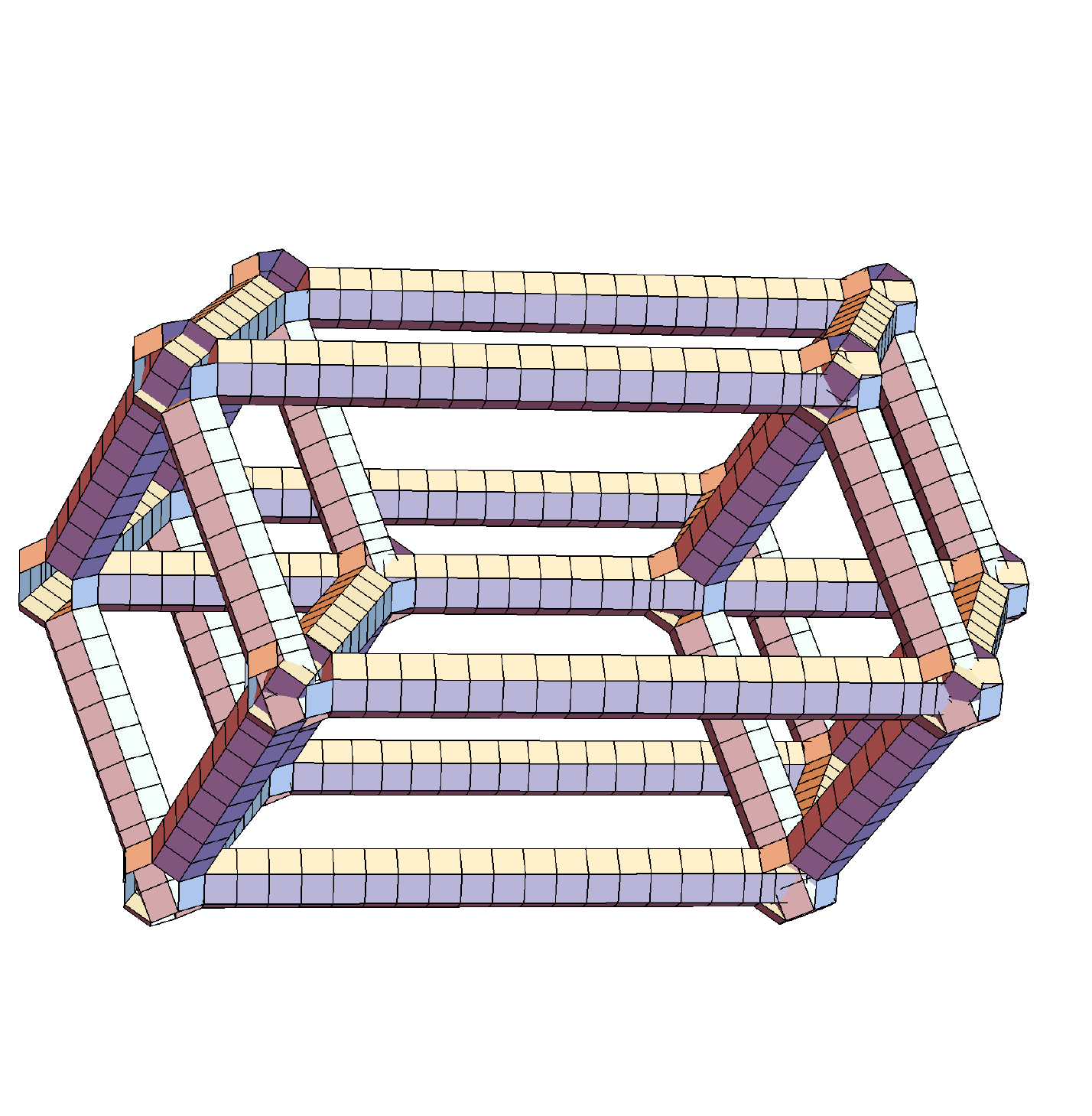}
\caption{A \((4)\)-RPS of genus 49}
\label{truncOctMol4d}
  \end{center}
\end{figure}

Figure \ref{truncCubOctMol4d} shows a genus 49 surface whose faces have degree four and eight. It is constructed out of truncated cuboctahedra placed at the vertices of a 4-cube and connected by hexagonal prisms. All hexagonal faces have been removed from the constituent truncated cuboctahedra and hexagonal prisms. The Euler characteristic of the surface is
\begin{align*}
\chi & = 16\cdot 2-16\cdot 8-64\cdot 6+64\cdot 6 = -96.
\end{align*} 

\begin{figure}[h]
  \begin{center}
    \includegraphics[width=\textwidth]{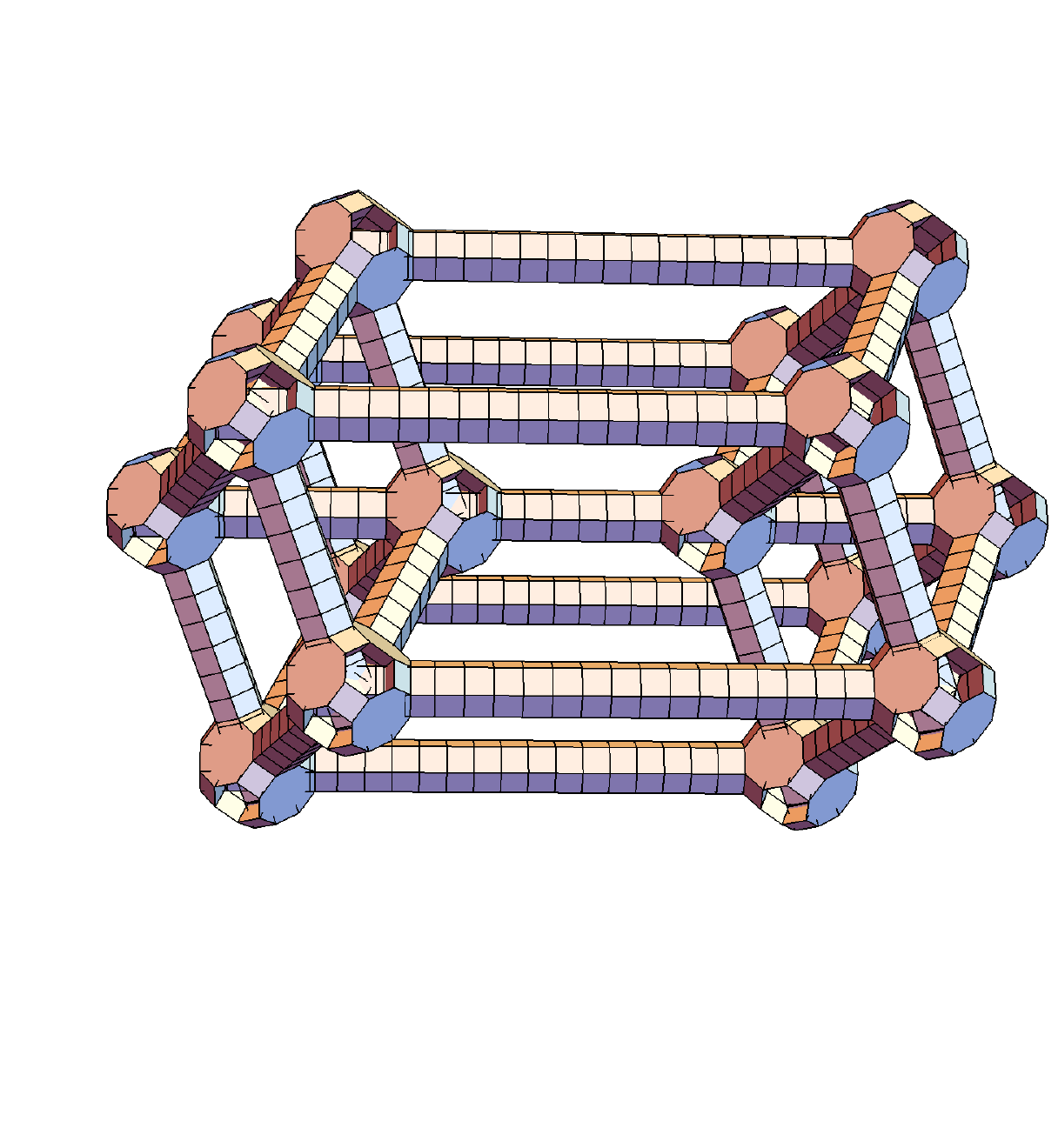}
\caption{A \((4,8)\)-RPS of genus 49}
\label{truncCubOctMol4d}
  \end{center}
\end{figure}

Figure \ref{truncCubOctMolSqHex} shows a genus 17 surface whose faces have degree four and eight. It is constructed out of truncated cuboctahedra placed at the vertices of a 3-cube and connected by octagonal prisms. All octagonal faces have been removed from the constituent truncated cuboctahedra and hexagonal prisms. The Euler characteristic of the surface is
\begin{align*}
\chi & = 8\cdot 2 - 8\cdot 6 -24\cdot 8+24\cdot 8 = -32.
\end{align*}

\begin{figure}[h]
  \begin{center}
    \includegraphics[width=\textwidth]{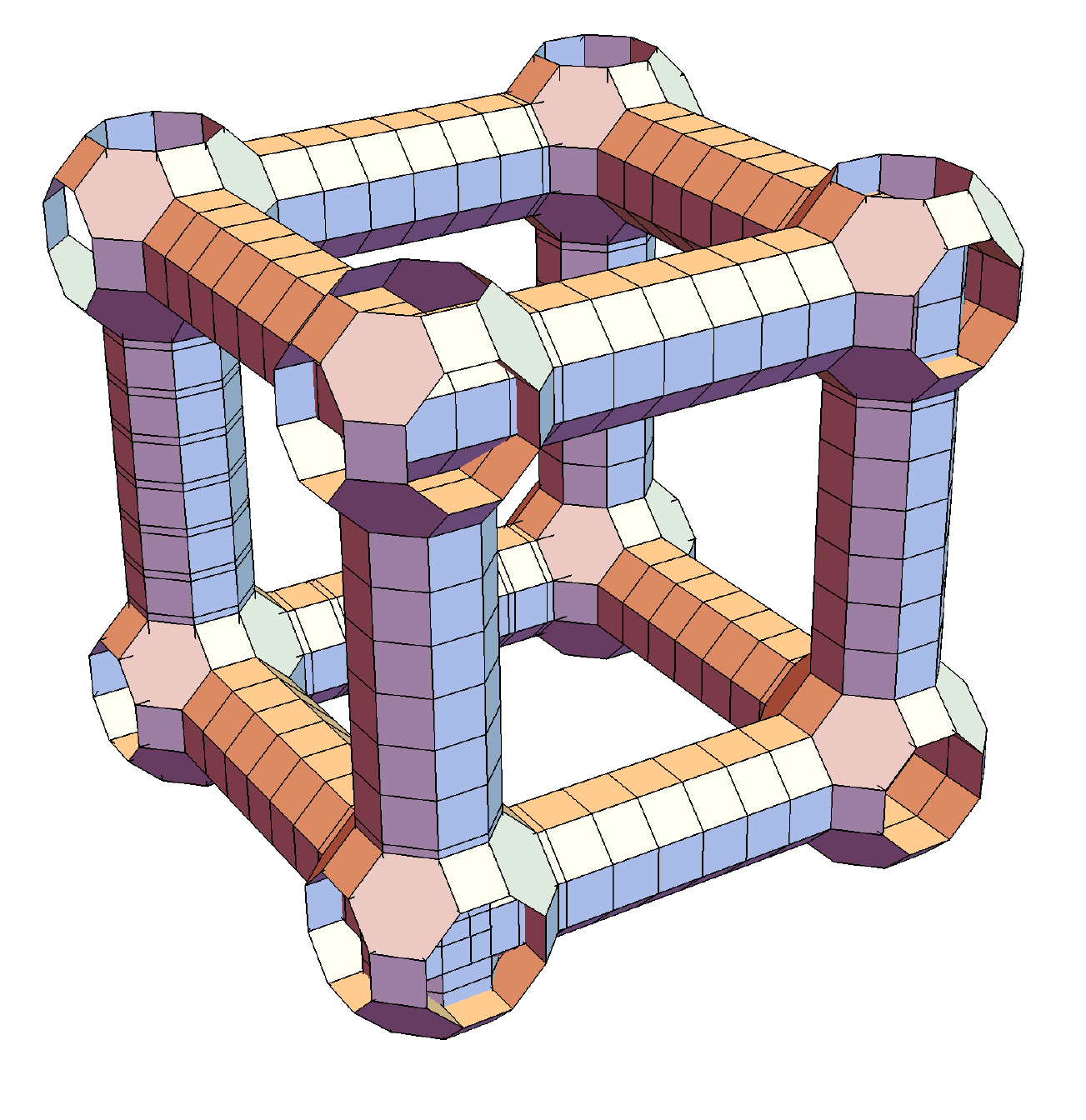}
\caption{A \((4,6)\)-RPS of genus 17}
\label{truncCubOctMolSqHex}
  \end{center}
\end{figure}


\subsection*{Acknowledgments}   
I thank Richard Kenyon for suggesting the study of RPSs and for his advice throughout this project. I thank Sanjay Ramassamy for suggesting a simplification of the proof in Section \ref{sec:rps5}. I also thank Ren Yi for many helpful conversations. 

\bibliographystyle{../hep}
\bibliography{../mybib}

\newcommand{\etalchar}[1]{$^{#1}$}
\begin{thebibliography}{OWHB17}

\bibitem[AB14]{ab2014}
J.~Ambjørn and T.~Budd, \textsl{ {Geodesic distances in Liouville quantum
  gravity}},
\newblock Nucl. Phys. \textbf{ B889}, 676--691 (2014), {1405.3424}.

\bibitem[Ale05]{a2005}
A.~D. Alexandrov,
\newblock \textsl{ Convex polyhedra},
\newblock Springer Monographs in Mathematics, Springer-Verlag, Berlin, 2005,
\newblock Translated from the 1950 Russian edition by N. S. Dairbekov, S. S.
  Kutateladze and A. B. Sossinsky, With comments and bibliography by V. A.
  Zalgaller and appendices by L. A. Shor and Yu. A. Volkov.

\bibitem[BCD{\etalchar{+}}02]{bcd2002}
T.~Biedl, T.~M. Chan, E.~D. Demaine, M.~L. Demaine, P.~Nijjar, R.~Uehara and
  M.~wei Wang,
\newblock Tighter bounds on the genus of nonorthogonal polyhedra built from
  rectangles,
\newblock in \textsl{ in: Proceedings of the 14th Canadian Conference on
  Computational Geometry, 2002}, pages 105--108, 2002.

\bibitem[Cox63]{c1963}
H.~S.~M. Coxeter,
\newblock \textsl{ Regular polytopes},
\newblock Second edition, The Macmillan Co., New York; Collier-Macmillan Ltd.,
  London, 1963.

\bibitem[Dav92]{d1992}
F.~David,
\newblock {Simplicial quantum gravity and random lattices},
\newblock in \textsl{ {Gravitation and quantizations. Proceedings, 57th Session
  of the Les Houches Summer School in Theoretical Physics, NATO Advanced Study
  Institute, Les Houches, France, July 5 - August 1, 1992}}, pages 0679--750,
  1992.

\bibitem[DO01]{do2001}
M.~{Donoso} and J.~{O'Rourke}, \textsl{ {Nonorthogonal Polyhedra Built from
  Rectangles}},
\newblock eprint arXiv:cs/0110059  (October 2001), {cs/0110059}.

\bibitem[GM63]{gm1963}
B.~Gr\"unbaum and T.~S. Motzkin, \textsl{ The number of hexagons and the
  simplicity of geodesics on certain polyhedra},
\newblock Canad. J. Math. \textbf{ 15}, 744--751 (1963).

\bibitem[Gol37]{g1937}
M.~Goldberg, \textsl{ A class of multi-symmetric polyhedra},
\newblock Tohoku Math. J \textbf{ 43}, 104--108 (1937).

\bibitem[Gru03a]{g2003a}
B.~Grunbaum,
\newblock Are your polyhedra the same as my polyhedra?,
\newblock in \textsl{ Discrete and computational geometry}, volume~25 of
  \textsl{ Algorithms Combin.}, pages 461--488, Springer, Berlin, 2003.

\bibitem[Gru03b]{g2003c}
B.~Grunbaum,
\newblock \textsl{ Convex polytopes}, volume 221 of \textsl{ Graduate Texts in
  Mathematics},
\newblock Springer-Verlag, New York, second edition, 2003,
\newblock Prepared and with a preface by Volker Kaibel, Victor Klee and Gunter
  M. Ziegler.

\bibitem[Gru09]{g2009}
B.~Grunbaum, \textsl{ An enduring error},
\newblock Elem. Math. \textbf{ 64}(3), 89--101 (2009).

\bibitem[GS09]{gs2009}
B.~Gr{\"u}nbaum and L.~Szilassi, \textsl{ Geometric realizations of special
  toroidal complexes},
\newblock Contrib. Discrete Math. \textbf{ 4}(1), 21--39 (2009).

\bibitem[GSW14]{gsw2014}
G.~G\'evay, E.~Schulte and J.~M. Wills, \textsl{ The regular {G}r\"unbaum
  polyhedron of genus 5},
\newblock Adv. Geom. \textbf{ 14}(3), 465--482 (2014).

\bibitem[JA68]{j1568}
W.~Jamnitzer and J.~Amman,
\newblock \textsl{ Perspectiva corporum regularium},
\newblock Vienna, 1568,
\newblock Available at
  \url{http://digital.slub-dresden.de/werkansicht/dlf/12830/45/}.

\bibitem[Ken96]{k1996n}
R.~Kenyon, \textsl{ A note on tiling with integer-sided rectangles},
\newblock J. Combin. Theory Ser. A \textbf{ 74}(2), 321--332 (1996).

\bibitem[LGM12]{lm2012}
J.-F. Le~Gall and G.~Miermont,
\newblock Scaling limits of random trees and planar maps,
\newblock in \textsl{ Probability and statistical physics in two and more
  dimensions}, volume~15 of \textsl{ Clay Math. Proc.}, pages 155--211, Amer.
  Math. Soc., Providence, RI, 2012.

\bibitem[OWHB17]{o2017}
J.~T. Overvelde, J.~C. Weaver, C.~Hoberman and K.~Bertoldi, \textsl{ Rational
  design of reconfigurable prismatic architected materials},
\newblock Nature \textbf{ 541}, 347--352 (2017).

\bibitem[Riv96]{r1996}
I.~Rivin, \textsl{ A characterization of ideal polyhedra in hyperbolic
  {$3$}-space},
\newblock Ann. of Math. (2) \textbf{ 143}(1), 51--70 (1996).

\bibitem[Sch11]{s2011}
R.~E. Schwartz,
\newblock \textsl{ Mostly surfaces}, volume~60 of \textsl{ Student Mathematical
  Library},
\newblock American Mathematical Society, Providence, RI, 2011.

\bibitem[{Sch}15]{s2015}
R.~E. {Schwartz}, \textsl{ {Notes on Shapes of Polyhedra}},
\newblock ArXiv e-prints  (June 2015), {1506.07252}.

\bibitem[Thu98]{t1998}
W.~P. Thurston,
\newblock Shapes of polyhedra and triangulations of the sphere,
\newblock in \textsl{ The {E}pstein birthday schrift}, volume~1 of \textsl{
  Geom. Topol. Monogr.}, pages 511--549, Geom. Topol. Publ., Coventry, 1998.

\end{thebibliography}

%
%
%
%

\end{document}